\documentclass[12pt]{article}
\usepackage{hyperref}
\usepackage{amssymb,amsmath,amsthm}
\usepackage{cite}
\usepackage{enumerate}
\usepackage{lmodern}
\usepackage{microtype}

\usepackage[letterpaper,hmargin=3.2cm,vmargin=3cm]{geometry}
\usepackage{setspace}
\makeatletter
\renewcommand{\section}{\@startsection%
{section}%
{1}%
{0em}%
{1.7em}%
{1.2em}%
{\normalfont\large\centering\bfseries}}
\renewcommand{\@seccntformat}[1]%
{\csname the#1\endcsname.\hspace{0.5em}}
\makeatother

\numberwithin{equation}{section}
\newtheorem{theorem}{Theorem}[section]
\newtheorem{proposition}{Proposition}[section]
\newtheorem{lemma}{Lemma}[section]
\newtheorem{corollary}{Corollary}[section]
\theoremstyle{definition}
\newtheorem{definition}{Definition}
\newtheorem{remark}{Remark}

\newcommand{\abs}[1]{\left|#1\right|}

\newcommand{\cc}[1]{\overline{#1}}

\newcommand{\ie}{\emph{i.\,e.\,}}

\newcommand{\cf}{\emph{cf.}}
\newcommand\CZ[2]{\pmb{\mathsf{ Z}}_{#1}({#2})}
\newcommand{\bea}{\begin{eqnarray}}
\newcommand{\eea}{\end{eqnarray}}
\newcommand{\beao}{\begin{eqnarray*}}
\newcommand{\eeao}{\end{eqnarray*}}

\newcommand{\llb}{\left\lbrace}
\newcommand{\rrb}{\right\rbrace}
\newcommand{\K}{\mathcal{K}}
\newcommand\R{{\mathbb R}}

\newcommand\N{{\mathbb N}}

\newcommand\C{{\mathbb C}}

\newcommand\D{{\mathbb D}}
\renewcommand{\H}{\mathcal{H}}
\newcommand{\T}{{T}}
\newcommand{\HH}{{\mathcal{H}\oplus\mathcal{H}}}
\newcommand\ip[2]{\langle {#1},{#2} \rangle}

\newcommand\no[1]{\| {#1} \|}

\def\fin{{\rm fin}}

\newcommand\Nk[2]{\pmb{\mathsf{N}}_{#1}({#2})}

\newcommand\oP[2]{{#1}\oplus {#2} }
\newcommand\oM[2]{{#1}\ominus {#2} }
\newcommand\cA[1]{\mathcal {#1}}

\newcommand\rE[1]{_{|_{#1}}}
\newcommand\vE[2]{{\begin{pmatrix}{#1}\\{#2}\end{pmatrix}}}
\DeclareMathOperator{\im}{Im\,}
\DeclareMathOperator{\re}{Re}
\DeclareMathOperator{\dom}{dom}
\DeclareMathOperator{\ran }{ran }
\DeclareMathOperator{\mul}{mul}

\DeclareMathOperator{\Span}{span}

\begin{document}
\begin{titlepage}
\title{Canonical decomposition of dissipative linear relations
\footnotetext{%
Mathematics Subject Classification(2010):
47A06,  
47B44, 
47A45, 
47A15.  
}
\footnotetext{%
Keywords:
Linear relations;
Dissipative relations;
Invariant and reducing subspaces;
Canonical decompositions.
}
}
\author{
\textbf{Josu\'e I. Rios-Cangas}
\\
\small Centro de Investigación en Matemáticas, A. C. \\[-1.6mm] 
\small Jalisco S/N, Col. Valenciana\\[-1.6mm] 
\small CP: 36023 Guanajuato, Gto, México.\\[-1.6mm]
\small \texttt{jottsmok@gmail.com}
}
\date{}
\maketitle
\vspace{-4mm}
\begin{center}
\begin{minipage}{5in}
  \centerline{{\bf Abstract}} \bigskip
  On the basis of  Sz. Nagy-Foia\c{s}-Langer and von
  Neumann-Wold decompositions, two decompositions for dissipative linear relations are given and they are realized 
by transforming invariant  subspaces for contractions, by means of the Z transform. These decompositions permit the
  separation of the selfadjoint and completely nonselfadjoint parts of
  a dissipative relation.
  \end{minipage}
\end{center}
\thispagestyle{empty}
\end{titlepage}
\section{Introduction}
\label{sec:intro}

This paper deals with dissipative linear relations in  Hilbert spaces, in which we are interested in decompose any closed dissipative relation 
into the selfadjoint and completely nonselfadjoint parts. In particular, decompose symmetric relations which do not admit proper dissipative  extensions. 

 The theory of linear relations is nowadays a widely abstract tool and of practical importance in extension theory, spectral analysis of dissipative and selfadjoint operators, canonical systems, and so forth. The concept of linear relation generalizes the notion of linear operator, when this is
identified by its graph (some works refer to a linear relation as
a multivalued linear operator \cf \cite{MR1631548}). Actually, a linear
relation is an operator whenever its multivalued part is the trivial space. we refer the reader to \cite{MR1631548,MR3178945,MR0123188,MR725424,MR2327982} for be familiar with the theory of linear relations and \cite{MR3057107,MR0361889,MR0477855,riossilva-expos,arXiv:1806.07503} for the theory of dissipative relations. 

The importance to studying dissipative linear operators arises in applications to problems in mathematical physics, since they are connected with dissipative systems \ie systems in which the energy is in general
nonconstant and nonincreasing in time (see for example \cite{MR2186014,MR3192426,MR3405898,MR3607473}). Phillips introduced the term of a dissipative operator in his seminal work \cite{MR0104919}, motivated to obtain a solution of the Cauchy problem for dissipative hyperbolic systems of partial differential equations.  Besides, he showed that a maximal dissipative operator generates a strongly continuous semi-group of contraction operators (see also \cite{MR3350727}). We draw the reader's attention to \cite{acta15367,MR2926200,MR3551233,MR3864399,MR3738818}, for more  applications of dissipative operators.

The theory of dissipative operators has its roots in the theory of contractions, \ie linear operators $T$ such that $\no T\leq 1$ (see \cite{MR0058128,MR2760647}, for an
exhaustive exposition of contractions). The class of contractions has been amply studied and is a well-understood class of operators (some generalizations of the concept of contraction can be found in \cite{MR1873624,MR0250106}). We would point out that a motivation for studying contractions stems from the invariant subspace problem \cite{MR675952,MR2003221,MR2760647}. Besides, contractions and dissipative operators are related with notable properties by virtue of the Cayley transform \cite[Chap.\, 4,\, Sec. 4]{MR2760647}. In order to take advantage of these eminent  properties, we use a small variation of the Cayley transform, named the Z transform (\cf \cite{MR2093073}).

The present work is concerned with a particular feature of
contractions, namely to the fact that they admit useful
decompositions. We focus our attention on two kinds of decompositions,
the Sz. Nagy-Foia\c{s}-Langer and the von Neumann-Wold
decompositions \cite{MR2760647,MR1464436} (see \cite{MR2681000} for a more general setting).  Our goal is to give a decomposition of  any closed dissipative relation, in which we isolate its selfadjoint part (see Theorem~\ref{Desconucnu01}). Particularly, in Theorem~\ref{descsisaem} we show that any symmetric relation, which does not admit dissipative proper extensions, is separated into its selfadjoint part and its elementary maximal part, \ie a relation whose Z transform is a unilateral shift. These decompositions are made by means of transforming invariant subspaces
for contractions. It is worth mentioning that the decompositions we show in this work, were intended as an application of boundary and quasi boundary triples, since on these theories appear by natural way the dissipative linear relations \cite{MR2953553,MR2150518,MR3050306,MR3705312}. On the other hand, these decompositions arise naturally in functional models of dissipative operators  
 \cite{MR0433235,a2018BBMMSN,MR1037765}.

The paper is organized as follows. In Section~\ref{sec:Invariant-reducing} we briefly recall some of the standard definitions on linear relations.  
 It is one of the main objectives in this section  to deal whit the problem of invariant and reducing subspaces for linear relations. Here, 
we show that the adjoint is distributed on reducing subspaces (see Theorem~\ref{teo:invariante-adjunto}).  Also, we show that linear relations of the form $\oP{\K}{\K}$, where $\K$ is a linear set, are invariant under the $Z$ transform (see
Remark~\ref{insub00}). A consequence of this is that the $Z$ transform
preserve reducing subspaces for any linear relation (See Theorem~\ref{redsubincal}). We deal in 
Section~\ref{sec:decomposition-of-relations}
with the general theory of contractions, in particular, the
Sz. Nagy-Foia\c{s}-Langer and the von Neumann-Wold
decompositions, in which the Sz. Nagy-Foia\c{s}-Langer
decomposition is extended to any closed contraction (see Theorem~\ref{cndfdes}). These results, together with the theory of reducing subspaces developed in the preceding section, are combined with the
theory of the $Z$ transform to obtain the required decomposition of any closed
dissipative relation.  Finally, as an illustrative example of the abstract techniques in this work, we present in Section \ref{sec:Example} a maximal symmetric relation, with nontrivial multivalued part as well as its corresponding decomposition.  

\subsection*{Acknowledgments} The author gratefully Acknowledge by the program for a global and integrated advancement of Mexican mathematics, project FORDECYT 265667, for which part of this research was realized. 

\section{Invariant and reducing subspaces for linear relations}
\label{sec:Invariant-reducing}
Let $(\H,\, \ip{\cdot}{\cdot})$ be a separable
Hilbert space, with inner product antilinear in its left
argument. We denote  $\HH$ as the orthogonal sum of two copies of the
Hilbert space $\H$ \cf \cite[Sec.\, 2.3]{MR1192782}. Throughout this
work, any linear set $T$ in $\HH$, is called a \emph{linear relation} (or merely relation), with 
\begin{align*}
 \dom \T:=\llb f\in \H\,:\ \vE fg\in T\rrb&\quad
  \ran\T:=\llb g\in \H\,:\ \vE fg\in T\rrb\\[1mm]
  \ker\T:=\llb f\in \H\,:\ \vE f0\in T\rrb&\quad
  \mul\T:=\llb g\in \H\,:\ \vE 0g\in T\rrb.
\end{align*} 

The concept of linear relation generalizes the notion of linear operator. Actually, a  relation $T$ is an operator (when it is identified by its graph) if and only its multivalued part $\mul T=\{0\}$.  

For two relations $T,S$ and $\zeta\in \C$, we denote the following linear relations: 
\beao
T+S&:=\llb\vE f{g+h}\ :\ \vE fg\in T,\ \ \vE fh\in S\rrb\quad
\zeta T:=\llb \vE f{\zeta g}\ :\ \vE fg\in T\rrb\\
ST&:=\llb \vE fk\ :\ \vE fg\in T,\ \ \vE gk\in S\rrb\qquad\T^{-1}:=\llb\vE gf\,:\, \vE fg\in T \rrb\,.
\eeao

The \emph{adjoint} of a relation $T$ is defined by
\begin{align*}
 \T^*:=\llb\vE hk\in \HH\ :\ \ip kf=\ip hg,\ \ \forall \vE fg\in \T\rrb,
\end{align*}
which turns out to be a closed relation with the following properties:
\begin{align}\label{poHs}
\T^*&=(-T^{-1})^{\perp},&S\subset  T&\Rightarrow T^*\subset S^*,\nonumber\\
\T^{**}&=\overline T,& (\alpha\T)^*&=\overline{\alpha}\T^*,\,\mbox{ with } \alpha\neq0,\\
(T^*)^{-1}&=(T^{-1})^*,&\ker \T^*&=(\ran \T)^{\perp}.\nonumber
\end{align}

Here and subsequently, a relation $T$ is \emph{bounded} whenever there exists a positive constant $C$ such that $\no g\leq C\,\no f$, for all $\vE fg\in T$. It is of our interest to point out that the boundedness definition for relations is not unique (see \cite{MR1631548}).  The boundedness condition we use in this paper implies that any bounded relation is an operator. 

We denote the \emph{quasi-regular} set of a linear relation $T$ by
\beao
\hat\rho(T):=\{\zeta\in\C\,:\,(T-\zeta I)^{-1}\,\text{ is bounded }\}\,,
\eeao
As in the case of operators, the quasi-regular set of a closed relation $T$ is open.

Let us consider the deficiency space of $T$, given by 
\beao
\Nk \zeta T:=\llb\vE f{\zeta f}\in T\rrb\,,\quad \zeta \in\C\,,
\eeao 
which is a bounded relation with $\dom \Nk \zeta T= \ker (T-\zeta I)$. Moreover, 
\bea\label{eq:constant-qregular}
\dim \Nk {\cc\zeta}{ T^*}\,,\quad \zeta\in\hat\rho(T)\,,
\eea
remains constant, on each connected component of $\hat\rho (T)$  (\cf \cite[Thm.\,3.7.4]{MR1192782}). 

The \emph{regular} set of a relation $T$, is defined by 
\beao
\rho(T):=\{\zeta\in\C\,:\, (T-\zeta I)^{-1}\in \cA B(\H)\}\,,
\eeao
 where  $\cA B(\H)$ denotes the class of bounded operators defined in the whole space $\H$. Also, the regular set is open and consists of all connected components of $\hat\rho(T)$, in which \eqref{eq:constant-qregular} is equal zero.
 
For a relation $T$, we consider the following sets:
\begin{align*}
\sigma(T)&:=\C\backslash \rho(T) &\mbox{(spectrum)}\\
\hat\sigma(T)&:=\C\backslash \hat\rho(T)&\mbox{(spectral core)}\\
\sigma_p(T)&:=\{\zeta \in \C\ :\ \dim \Nk\zeta T\neq 0\}&\mbox{(point spectrum)}\\
\sigma_c(T)&:=\{\zeta \in \C\ :\ \ran (T-\zeta I)\neq \overline{\ran (T-\zeta I)}\}&\mbox{(continuous spectrum)}\\
\sigma_{r}(T)&=\sigma(T)\backslash\hat\sigma(T)&\mbox{(residual spectrum)}
\end{align*}
 Analogously to the case of operators, it fulfills  \bea\label{eq:kerspec}
\sigma_{p}(T)\cup\sigma_{c}(T)=\hat\sigma(T)\,.\eea
\begin{remark}\label{re:complex-conjugate}
For a closed relation $T$, one has that $\sigma(T^*)$ is the complex conjugate of $\sigma(T)$ \cite[Prop.\,2.5]{riossilva-expos}. The same holds for the continuous spectrum, since $\ran (T-\zeta I)$ and $\ran (T^*-\cc\zeta I)$ are simultaneously closed (see \cite[Lem.\,2.3]{MR0361889}).
\end{remark}\label{prop:complex-residual-spectrum}
\begin{proposition}\label{prop:conjugate-residual}
Let $T$ be a closed relation. If $\zeta\in\sigma_r(T)$ then $\cc\zeta\in\sigma_p(T^*)\backslash\sigma_c(T^*)$.
\end{proposition}
\begin{proof}
Since $\zeta\in\sigma_r(T)$, one has  that $(T-\zeta I)^{-1}$ is closed and bounded, which is not defined on the whole space. In this fashion, $\ran (T-\zeta I)$ is closed as well as $\ran (T^*-\cc\zeta I)$ and 
\beao
\ker(T^*-\cc\zeta I)=[\ran (T-\zeta I)]^\perp\neq\{0\}\,.
\eeao
These facts imply the required.
\end{proof}

Before we proceed to the theory of invariant subspaces, we shall set the following. For a relation $T$ in $\HH$ and a linear set $\K$ in $\H$, we denote 
\begin{align}\label{treshH}
T_{\K}:=T\cap(\oP{\K}{\K})\,,
\end{align}
where $\oP{\K}{\K}$ represents the orthogonal sum of
two copies of $\K$. It is clear that  $T_{\H}=T$ and $T_{\{0\}}=\oP{\{0\}}{\{0\}}$.

\begin{definition}\label{defreV}Let $T$ be a relation in $\HH$. 
A subspace $\K\subset\H$ is called \emph{invariant} for $T$ 
  (briefly $T$-invariant), when the following conditions are true:
\begin{enumerate}[(i)]
\item \label{domeH} $\dom T=\oP{(\dom T\cap{\K})}{(\dom T\cap{\K^{\perp}})}$.
\item $\mul T=\oP{(\mul T\cap \K)}{(\mul T\cap \K^{\perp})}$.
\item $\dom T_{\K}=\dom T\cap{\K}$.
\end{enumerate}
\end{definition}
We see at once  that  $\H$ and $\{0\}$ are invariant, for any linear relation.

\begin{definition}\label{deredT}
For a relation $T$ in $\HH$, a subspace $\K\subset \H$ is said \emph{to reduce} $T$ if  
\beao
T=\oP{T_{\K}}{T_{\K^{\perp}}}\,
\eeao
\end{definition}
The subspaces $\K$ and $\K^{\perp}$ reduce $T$ simultaneously and in this case 
\begin{gather}\label{eq:deco-redu-space}
\begin{split}
\dom T=\oP{\dom T_{\K}}{\dom T_{\K^{\perp}}}\,,\quad \ker T=\oP{\ker T_{\K}}{\ker T_{\K^{\perp}}}\,,\\
\ran T=\oP{\ran T_{\K}}{\ran T_{\K^{\perp}}}\,,\quad \mul T=\oP{\mul T_{\K}}{\mul T_{\K^{\perp}}}\,.
\end{split}
\end{gather}

 \begin{remark}\label{re:Decom-into-relations} The existence of relations $T_{1}\subset\oP{\K}{\K}$ and
   $T_{2}\subset\oP{\K^{\perp}}{\K^{\perp}}$, in which 
 \beao
   T=\oP{T_{1}}{T_{2}}\,,\eeao
 implies that $\K$ reduces $T$, with  $T_{1}=T_{\K}$ and $T_{2}=T_{\K^{\perp}}$.   
 \end{remark}

\begin{proposition}\label{eqredin}
A subspace $\K$ reduces  $T$ if and only if $\K$ and $\K^{\perp}$ are  $T$-invariant.
\end{proposition}
\begin{proof}
  If $\K$ reduces $T$, then by verifying the inclusions in both
  directions, one arrives at 
\beao
\begin{split}
\dom T_{\K}=\dom T\cap\K\,,\quad \ran T_{\K}=\ran T\cap\K\,,\\
\ker T_{\K}=\ker T\cap\K\,,\quad \mul T_{\K}=\mul T\cap\K\,.
\end{split}
\eeao
Therefore,  one has by \eqref{eq:deco-redu-space} that $\K$ is 
$T$-invariant. This also holds for $\K^\perp$, since it reduces $T$.

The converse follows once
we show that $T\subset \oP{T_{\K}}{T_{\K^{\perp}}}$. By inclusion, 
let $\vE fg\in T$ and the first condition of $T$-invariant implies that  there exist  \bea\label{eq:elements-of-inv0} \vE
as\in T_{\K};\quad \vE bt\in T_{\K^{\perp}}\,, \eea such that
$f=a+b$. Thus $\vE f{s+t}\in T$,  which yields
$\vE 0{g-(s+t)}\in T$. The  second condition of $T$-invariant indicates the  existence of  
\bea\label{eq:elements-of-inv1} \vE 0h\in T_{\K}\,;\quad \vE 0k\in
T_{\K^{\perp}}, \eea such that $g-(s+t)=h+k$. Hence,
\eqref{eq:elements-of-inv0} and \eqref{eq:elements-of-inv1} produce 
\begin{align}
\vE fg=\vE a{s+h}+\vE b{t+k}\in \oP{T_{\K}}{T_{\K^{\perp}}}\,,
\end{align}
as required. 
\end{proof}

Let us note the following property, which is required in the next assertion.  If $\K$ reduces $T$, then by a simple computation, one obtains  \bea\label{eq:closedT-Tinvariant} \cc
T=\oP{\cc T_{\K}}{\cc T_{\K^{\perp}}}\,.\eea This implies  that $T$ is
closed if and only if $T_{\K}$ and $T_{\K^{\perp}}$ are closed.

\begin{theorem}\label{teo:invariante-adjunto}
If $\K$ reduces  $T$, then  $\K$ reduces $T^{*}$ and the following holds 
\bea\label{eq:invariante-adjunto}
(\oP{T_{\K}}{T_{\K^{\perp}}})^{*}=\oP{(T_{\K})^{*}}{(T_{\K^{\perp}})^{*}}\,.
\eea
\end{theorem}
\begin{proof}
By hypothesis $T=\oP{T_{\K}}{T_{\K^{\perp}}}$. Besides, the first property of \eqref{poHs} yields 
\begin{gather}\label{eq:decomp-Tk}
\oP{-(\cc T_{\K})^{-1}}{(T_{\K})^{*}}=\oP\K\K;\quad
\oP{-(\cc T_{\K^{\perp}})^{-1}}{(T_{\K^{\perp}})^{*}}=\oP{\K^{\perp}}{\K^{\perp}}\,.
\end{gather}
Thus, in view of \eqref{eq:closedT-Tinvariant} and \eqref{eq:decomp-Tk},
\begin{align*}
\oP{-(\cc T)^{-1}}{[(\oP{T_{\K})^{*}}{(T_{\K^{\perp}}})^{*}]}&=\oP{-[\oP{\cc T_{\K}}{\cc T_{\K^{\perp}}}]^{-1}}{[(\oP{T_{\K})^{*}}{(T_{\K^{\perp}}})^{*}]}\\
&=\oP{[\oP{-(\cc T_{\K})^{-1}}{(T_{\K})^{*}}]}{[\oP{-(\cc T_{\K^{\perp}})^{-1}}{(T_{\K^{\perp}})^{*}}]}\\
&=\oP{(\oP\K\K)}{(\oP{\K^{\perp}}{\K^{\perp}})}\\&=\HH=\oP{-(\cc T)^{-1}}{T^{*}}\,.
\end{align*}
whence one arrives at \eqref{eq:invariante-adjunto}. Moreover, \eqref{eq:decomp-Tk} also implies that $(T_{\K})^{*}\subset \oP{\K}{\K}$ and
$(T_{\K^{\perp}})^{*}\subset \oP{\K^{\perp}}{\K^{\perp}}$. Hence from Remark \eqref{re:Decom-into-relations}, 
 $\K$ reduces $T^{*}$. 
\end{proof}
Following  the last result, if $\K$ reduces $T$ then 
\bea\label{eq:reducing-distribution}
(T_{\K})^{*}=(T^*)_{\K};\quad (T_{\K^{\perp}})^{*}=(T^*)_{\K^{\perp}}\,.
\eea

We shall introduce a version of the Cayley transform for linear relations (\cf \cite{MR2093073}).

\begin{definition}
  For a relation $T$ and $\zeta \in \C$,
 we define the \emph{$Z$ transform} of $T$ by \beao \CZ \zeta T:=\llb \vE
  {g-\overline{\zeta}f}{\cc\zeta g-|\zeta|^{2} f}\ :\ \vE fg \in T\rrb\,.
  \eeao
\end{definition}
The $Z$ transform is a linear relation with 
\begin{align}\label{CtoT}
\begin{split}
 \dom \CZ \zeta T=\ran (T-\overline{\zeta}I)\,,& \quad
 \ran \CZ \zeta T=\ran (T-\zeta I)\,,\\ \mul \CZ \zeta T=\ker
 (T-\overline{\zeta} I)\,,& \quad \ker \CZ \zeta T=\ker (T-\zeta I)\,.
 \end{split}
\end{align}Moreover, the following properties hold (see \cite[Lems.\,2.6,
2.7]{MR0361889} and \cite[Props.\,3.6, 3.7]{MR2093073}).
For any $\zeta \in\C$:
\begin{enumerate}[{(i)}]
 \item \label{cay01} $\CZ \zeta {\CZ \zeta T}=T$.
 \item $\CZ \zeta T\subset \CZ \zeta S\ \Leftrightarrow\ T\subset S$.\label{tres}
 \item $\CZ {-\zeta} T=-\CZ \zeta {-T}$.
 \item $\CZ \zeta {T^{-1}}=\CZ {\overline {\zeta}} T=(\CZ \zeta {T})^{-1}$, if $\abs{z}=1$. 
  \end{enumerate}
For any $\zeta \in \C\backslash \R$:
  \begin{enumerate}[{(i)}]
  \setcounter{enumi}{4}
 \item $\CZ { \zeta} {T\dotplus S}=\CZ { \zeta} T\dotplus \CZ { \zeta} S$.
 \item \label{cay07} $\CZ{\pm i}{\oP TS}=\oP{\CZ {\pm i} T}{\CZ {\pm i} S}$.
 \item\label{cay05} $\CZ { \zeta} {T^*}=(\CZ {\overline{\zeta}} {T})^{*}$.
 \item\label{cay10} $\cc{\CZ \zeta T}=\CZ \zeta {\cc T}$.
\end{enumerate}

\begin{remark}\label{insub00}
For any linear set $\K\subset\H$, the following holds:
 \begin{align}\label{eqsubes}
\CZ {\zeta}{\oP{\K}{\K}}=\oP{\K}{\K}\,\quad(\zeta\in\C)\,.
\end{align}
Indeed, it is straightforward to see that  $\CZ {\zeta}{\oP{\K}{\K}}\subset\oP{\K}{\K}$ and the other inclusion follows applying the property \eqref{cay01} of the $Z$ transform. 
\end{remark}

\begin{theorem}\label{redsubincal}
  A subspace $\K$ reduces $T$ if and only if it reduces
  $\CZ{\pm i} T$ (the assertion is meant to hold separately for $+i$
  and $-i$).
\end{theorem}
\begin{proof}
If  $\K$ reduces $T$, then  $T=\oP{T_{\K}}{T_{\K^{\perp}}}$ and   thus
\begin{align*}
\CZ{\pm i} T=\oP{\CZ{\pm i}{T_{\K}}}{\CZ{\pm i}{T_{{ \K}^{\perp}}}}\,.
\end{align*}
Since $T_{\K}\subset\oP{\K}{\K}$ and $T_{\K^{\perp}}\subset\oP{\K^{\perp}}{\K^{\perp}}$, one has by \eqref{eqsubes} that 
\begin{align*}
\CZ{\pm i}{T_{\K}}&\subset\oP{\K}{\K};\\
\CZ{\pm i}{T_{\K^{\perp}}}&\subset\oP{\K^{\perp}}{\K^{\perp}}\,.
\end{align*}
Therefore $\K$ reduces  $\CZ{\pm i}T$. The converse follows replacing $T$ by $\CZ{\pm i}T$, in the reasoning above. 
\end{proof}

\section{The canonical decomposition of dissipative relations}
\label{sec:decomposition-of-relations}

We begin this section with a brief exposition on general concepts of  contractions. We recall that a linear operator $V$ in $\HH$ (seen as a linear relation) is a \emph{contraction} if
it is bounded with $\no V\leq 1$. Particularly, $V$ is an \emph{isometry} if $V^{-1}\subset V^{*}$ or \emph{unitary} whenever $V^{-1}=V^{*}$. In both cases its norm is equal one.

We say that $V$ is a \emph{maximal contraction} if it does not admit proper contractive extensions. This property is equivalent to say that $V$ belongs to the class  $\cA B(\H)$.

\begin{definition}
  A contraction $V$ is said to be \emph{completely nonunitary} (c.n.u. for short), when there is no nonzero reducing subspace $\K$
  for $V$, in which $V_{\K}$ is  unitary.
\end{definition}

The following result is an extension of the so-called
Sz. Nagy-Foia\c{s}-Langer decomposition (see \cite[Chap.\,I,\,Sec.\,3,\,Thm.\, 3.2]{MR2760647}), which
is proven for contractions in $\cA B(\H)$.

\begin{theorem}\label{cndfdes}
  For every closed contraction $V$, there exists a unique reducing
  subspace $\K$ for $V$, such that $V_{\K}$ is unitary
  and $V_{\K^{\perp}}$ is completely nonunitary.
\end{theorem}
\begin{proof}
We begin by denoting 
\begin{gather}\label{eq-hatVinucnu}
\hat V:=\oP{V}{W}\,,\quad\mbox{where}\quad W=\llb \vE h0\ :\ h\in \oM\H{\dom V}\rrb\,.
\end{gather}
Inasmuch as $\dom V$ is closed, $\hat V$ is a maximal contraction. Then, the
Sz. Nagy-Foia\c{s}-Langer decomposition asserts that there exists a unique reducing  subspace $\K$ for $\hat V$,  such that  
 $\hat V_{\K}$ is unitary and $\hat V_{\K^{\perp}}$ is c.n.u. Thus, for every $\vE fg\in \hat V_{\K}\subset \hat V$, in view of
\eqref{eq-hatVinucnu}, there is $\vE {f_{1}}g\in V$ and $f_{2}\in (\dom V)^{\perp}$, such that   
$ f=f_{1}+f_{2}$. Thereby,
\begin{align*}
 \no {f_{1}}^{2}\geq\no g^{2}=\no f^{2}= \no {f_{1}+f_{2}}^{2}=\no {f_{1}}^{2}+\no {f_{2}}^{2}\,,
\end{align*}
wherefrom  $f_{2}=0$.  Hence,  $ \hat V_{\K}\subset V_{\K}$ and they are the same, since $V\subset \hat V$. The previous reasoning implies that $\K$ reduces $V$ as well as $W\subset \hat V_{\K^{\perp}}$. Furthermore,  
$V_{K^{\perp}}= \oM{\hat V_{\K^{\perp}}}{W}$, which is a
c.n.u. contraction. The uniqueness follows directly, bearing in mind that a  reducing subspace for $V$, also reduces $\hat V$.
\end{proof}

In what follows, we will turn our attention to a particular class of isometries, known as unilateral shifts.

For an isometry $V$ in $\cA B(\H)$, we say that a subspace $\cA L\subset\H$ is a {\it wandering} space for $V$, if  for all $n,m\in\N\cup\{0\}$, with $n\neq m$,  
\begin{align*}
V^m\cA L\perp V^n\cA L\,. 
\end{align*}

\begin{definition}\label{def:desunilateral}
  An isometric operator $V$ in $\cA B(\H)$ is called  \emph{unilateral
  shift}, if there exists a wandering space $\cA L$ for $V$, 
  such that 
  \begin{gather}\label{eq:Vnwandering}
\cA L\oplus V\cA L\oplus V^2\cA L\oplus\dots=\H\,.\end{gather}
 \end{definition}
The wandering space for a unilateral shift $V$ is uniquely determined by 
\beao\cA L=\oM{\H}{\ran V}\,.\eeao
Besides, it is straightforward to computes  that 
\beao
V^*=\oP{V^{-1}}{\llb\vE{l}0\,:\,l\in\cA L\rrb}\,.
\eeao

Let us introduce the following assertion which is well-known as the von Neumann-Wold decomposition \cite[Chap.\,I,\,Sec.\,1,\,Thm.\, 1.1]{MR2760647}.

\begin{theorem}\label{Desconucnu}
  For every isometric operator $V$ in $\cA B(\H)$, there exists a
  unique reducing subspace $\K$ for $V$, such that $V_{\K}$ is unitary
  and $V_{\K^{\perp}}$ is a unilateral shift. Namely,
\begin{gather}\label{carHhanm}
  \K:=\bigcap_{n=0}^\infty \ran V^{n}\quad\mbox{ then }
  \quad \K^{\perp}=\bigoplus_{n=0}^\infty V^n\cA L\,, \quad \mbox{ where }\quad \cA L=\oM{\H}{\ran V}\,.
\end{gather}
\end{theorem}
In the last result, the space $\K$ may be trivial or the whole space.  

\begin{corollary}\label{coro:cnuunishift}
An isometric operator in $\cA B(\H)$ is a unilateral shift if and only if it is completely nonunitary.
\end{corollary}
\begin{proof}
Assume $V$ is a unilateral shift and suppose that $\K$ is a reducing subspace for $V$, in which $V_\K$ is unitary. Then $\K=V^n\K\subset V^n\H$, for $n=0,1,2,\dots$ On the other hand, in view of \eqref{eq:Vnwandering}, the wandering space $\cA L$ for $V$, satisfies 
\beao
V^n\cA L=\oM{V^n\H}{V^{n+1}\H}\,.
\eeao
Consequently,  $\K\perp V^n\cA L$, for all $n\in\N\cup\{0\}$. Thus  \eqref{eq:Vnwandering} implies $\K=\{0\}$ and hence $V$ is c.n.u. The converse follows from Theorem~\ref{Desconucnu}.
\end{proof}

Theorems~\ref{cndfdes} and~\ref{Desconucnu} present  two kind of  decompositions for closed contraction, uniquely determined by  their unitary and completely nonunitary parts.  Hereinafter, we will address the dissipative relations to give the counterpart to these theorems.

\begin{definition}
A relation $L$ is called \emph{dissipative} if for all $\vE fg\in L$, 
\beao\im \ip fg\geq0\,.\eeao
In particular, $L$ is \emph{symmetric} if $L\subset L^*$ and \emph{selfadjoint} when $L=L^*$. Furthermore, $L$ is say to be \emph{maximal dissipative} if it does not have proper dissipative extensions. 
\end{definition}

For the reader's convenience, the following result from \cite{riossilva-expos}
 is brought up.
\begin{proposition}\label{caydac}
 If $\zeta$ is in the upper half plane $\C_{+}$, such that $\abs{\zeta}=1$, then  a 
  linear relation $L$ is (closed, maximal) dissipative (symmetric,
  selfadjoint) if and only if $\CZ\zeta L$ is a (closed, maximal)
  contraction (isometry, unitary).
\end{proposition}
 The above assertion clarifies that the $Z$
 transform gives a one-to-one correspondence between contractions and
 dissipative relations.
 
\begin{definition}We call a dissipative relation $L$ \emph{completely nonselfadjoint}
(briefly c.n.s.), if there is no nonzero reducing subspace
$\K$ for $L$, in which $L_{\K}$ is selfadjoint.
\end{definition}

\begin{remark}\label{re:c.n.s-operator}
If a closed dissipative relation $L$ is c.n.s., then it is an operator. Indeed, since $\dom L\subset (\mul L)^\perp$ (\cf \cite[Sec.\,2]{MR3057107}), one has that $\mul L$ is a reducing subspace for $L$, in which $L$ is selfadjoint. Hence, $\mul L=\{0\}$. 
\end{remark}

\begin{proposition}\label{cnucna}
 A relation $L$ is a completely nonselfadjoint, dissipative relation if and only if $V=\CZ i L$ is a
  completely nonunitary, contraction.
\end{proposition}
\begin{proof}
We first suppose that $L$ is a c.n.s. dissipative relation. By Proposition~\ref{caydac}, one has that $V=\CZ i L$ is a contraction. Besides, if 
there exists a nonzero reducing subspace $\K$ for $V$, such that  $V_{\K}$ is
unitary, then Theorem~\ref{redsubincal} implies that
$\K$ also reduces $L$ and Proposition~\ref{caydac} states that  $\CZ i{V_{\K}}\subset L$ is selfadjoint. This contradicts our assumption that $L$ is c.n.s. Therefore, $V$ is c.n.u. The proof of the converse is handled in the same lines.
\end{proof}
We shall continue with the analogue of the Sz. Nagy-Foia\c{s}-Langer
decomposition for dissipative relations.

\begin{theorem}\label{Desconucnu01} 
If $L$ is a closed dissipative relation, then there exists a unique
  reducing subspace $\K$ for $L$, such that $L_{\K}$ is selfadjoint and $L_{\K^{\perp}}$ is completely nonselfadjoint.
\end{theorem} 
\begin{proof}
Inasmuch as $L$ is a closed dissipative relation, then by Proposition~\ref{caydac}, one has that $\CZ i L$ is a
  closed contraction. Thus, in according to  Theorem~\ref{cndfdes}, there exists a unique reducing 
  subspace $\K$ for $\CZ i L$, in which $\CZ i L_{\K}$ is
  unitary and $\CZ i L_{\K^{\perp}}$ is c.n.u. In this fashion, Theorem~\ref{redsubincal} implies that $\K$ reduces $L$ and besides, 
\begin{align}\label{eq:properties-Z-Nagy}
\begin{split}
L&=\CZ i {\CZ i L}\\
&=\CZ i {\oP{\CZ i L_{\K}}{\CZ i L_{\K^{\perp}}}}\\
&=\oP{\CZ i {\CZ i L_{\K}}}{\CZ i {\CZ i L_{\K^{\perp}}}}\,.
\end{split}
\end{align}
Concluding, from Remark~\ref{re:Decom-into-relations} and Proposition~\ref{caydac}, one has  that  $L_{\K}=\CZ i
{\CZ i L_{\K}}$ is selfadjoint  and  $L_{\K^{\perp}}=\CZ i {\CZ i L_{\K^{\perp}}}$  is c.n.s., by virtue of  Proposition~\ref{cnucna}.

To prove the uniqueness. If $\K'$ holds the same properties of $\K$, for $L$. Then, Theorem~\ref{redsubincal} asserts that $\K'$ also reduces
\begin{align*}
\CZ i L&=\CZ i {\oP{L_{\K'}}L_{\K'^{\perp}}}\\
&=\oP{\CZ i{L_{\K'}}}{\CZ i{L_{\K'^{\perp}}}}\,.
\end{align*}
Moreover, again by Remark~\ref{re:Decom-into-relations} and Propositions~\ref{caydac},~\ref{cnucna}, one obtains that $\K'$ satisfies the same properties of $\K$, for $\CZ i L$. Hence  $\K'=\K$, since $\K$ is unique for $\CZ i L$.  
\end{proof}
As consequence of the last result and Remark~\ref{re:c.n.s-operator}, the multivalued part of a closed dissipative relation only can inside in its selfadjoint part.  

To present the other decomposition, we shall work in the class of symmetric relations. We follow  \cite{riossilva-expos} in assuming that the core spectrum of a symmetric relation $A$, satisfies 
$\hat\sigma(A)\subset\R$. Furthermore, if $A$ is maximal then $\sigma(A)\subset \C_+\cup\R$.  
\begin{remark}\label{rem:core-spectra-nsa}
Taking into account \eqref{eq:kerspec}, if $A$ is a completely nonselfadjoint, symmetric relation, then  $\sigma_c(A)=\hat \sigma(A)$. Indeed, the linear envelope of every eigenvector of $A$ is a reducing subspace for $A$, in which $A$ is selfadjoint.
\end{remark} 

\begin{definition} A symmetric relation $A$ is called \emph{elementary maximal}, if its transform $\CZ i A$ is a unilateral shift (\cf \cite[Sec.\,82]{MR1255973}).
\end{definition}

\begin{remark}\label{re:pointspectra-c.n.s.}
An elementary maximal relation $A$ is actually maximal, since the unilateral shifts are maximal. This involves that    $\dim \Nk {\cc\zeta}{A^*}=0$, for all $\zeta\in\C_-$ (\cf \cite{riossilva-expos}). Thus, the first von Neumann formula for relations (see for instance \cite[Thm.\,6.1]{MR0361889}) implies that 
\begin{gather}\label{eq:adjoint-elementary}
A^*=A\dotplus\Nk{\zeta}{A^*}\,,\quad\zeta\in\C_-\,,
\end{gather}
where for $\zeta=-i$, the direct sum turns to be orthogonal.
\end{remark} 
\begin{lemma}\label{cnaelmax}
A maximal symmetric relation is  elementary maximal if and only if it is completely nonselfadjoint. 
\end{lemma}
\begin{proof}
It follows straightforward from  Corollary~\ref{coro:cnuunishift} and Propositions~\ref{caydac},~\ref{cnucna}.
\end{proof}

The following assertion uses the fact that a maximal dissipative relation $L$ satisfies $\cc{\dom L}=(\mul L)^\perp$ (\cf \cite[Lem.\,2.1]{MR3057107}).

\begin{theorem}\label{teo:elem-maximal}
If  $A$ is an elementary maximal relation, then $A$ is an unbounded densely defined operator, with the following  spectral properties:
\bea\label{eq:spectral-properties-ME}
\begin{split}
\sigma_p(A)=\emptyset\,,\quad \sigma_c(A)=\R\,,\quad \sigma_r(A)=\C_+\,,\\
\sigma_p(A^*)=\C_-\,,\quad \sigma_c(A^*)=\R\,,\quad \sigma_r(A^*)=\emptyset\,.
\end{split}
\eea
\end{theorem}
\begin{proof}
Since $A$ is maximal, then it is closed. Besides, A is an operator, in according to  Lemma~\ref{cnaelmax} and Remark~\ref{re:c.n.s-operator}. This fact also implies that $\dom A$ is dense but not the whole space, otherwise $A$ has to be  selfadjoint, but this contradicts Lemma \ref{cnaelmax}. Also, in this way, one obtains that $A$ is unbounded.

We now proceed to show \eqref{eq:spectral-properties-ME}. Since $A$ is c.n.s, it is straightforward to see from Remark~\ref{rem:core-spectra-nsa} that $\sigma_p(A)=\emptyset$. Additionally, Proposition~\ref{prop:conjugate-residual} implies $\sigma_r(A^*)\subset \sigma_p(A)=\emptyset$. On the other hand, \eqref{eq:adjoint-elementary} yields $\dim \Nk\zeta {A^*}\neq 0$, for all $\zeta\in\C_-$, which  means $\C_-\subset\sigma_p(A^*)$. To show the other inclusion, if $\zeta\in\sigma_p(A^*)$, then there exists $\vE f{\zeta f}\in A^*$, with $\no f=1$. Again by \eqref{eq:adjoint-elementary}, there is $\vE hk\in A$ and $\vE t{-it}\in A^*$, such that 
\bea\label{eq:elements-of-A*}
\vE f{\zeta f}=\vE hk+\vE t{-i t}\,.
\eea
Observe that $t\neq0$, since $\sigma_p(A)=\emptyset$. Moreover, by virtue of  $A$ is symmetric, one produces $\ip hk\in\R$ and $\ip kt=-i\ip ht$. Thus, taking into account  \eqref{eq:elements-of-A*},
 \begin{align*}
 \im \zeta&=\im \ip f{\zeta f}\\&=\im (\ip hk+2\re\ip tk-i\no t^2)<0\,.
 \end{align*}
 This proves $\sigma_p(A^*)\subset C_-$ and hence they are equals. 
 
 The maximality of $A$ implies that $\sigma(A)\subset\C_+\cup\R$. 
Consequently, by Remark~\ref{re:complex-conjugate} and since 
 $\sigma_r(A^*)=\emptyset$, one has 
 $\hat\sigma(A^*)=\sigma(A^*)\subset\C_-\cup\R$. Then, $\hat\sigma(A^*)=\C_-\cup\R$, since 
 $\hat\sigma(A^*)$ is closed and contains the lower half plane. Hence, by virtue of \eqref{eq:kerspec}, one obtains $\sigma_c(A^*)=\R$. To conclude, again Remark~\ref{re:complex-conjugate} yields $\sigma_c(A)=\R$ and $\sigma(A)=\C_+\cup\R$, which asserts $\sigma_r(A)=\C_+$. 
\end{proof}

The method used in the last proof can be carried over to unilateral shift operators, holding similar properties to \eqref{eq:spectral-properties-ME} in the following sense: If $V$ is a unilateral shift, then:
\begin{gather*}
\sigma_p(V)=\emptyset\,,\quad \sigma_c(V)=\partial\D\,,\quad \sigma_r(V)=\D\,,\\
\sigma_p(V^*)=\D\,,\quad \sigma_c(V^*)=\partial\D\,,\quad \sigma_r(V^*)=\emptyset\,,
\end{gather*}
where $\D$ is the open unit disc and $\partial\D$ its boundary.

We conclude this section by showing the analogue of the von Neumann-Wold
decomposition for symmetric relations. 

\begin{theorem}\label{descsisaem}
If $A$ is a maximal symmetric relation, then 
\begin{gather}\label{eq:reducing-maximal-symm}
\K=\bigoplus_{n=0}^\infty \CZ iA^n\cA L\,, \quad \mbox{ with }\quad \cA L=\dom \Nk {-i}{A^*}\,,
\end{gather}
is the unique reducing subspace for $A$, such that $A_{\K}$ is elementary maximal and $A_{\K^\perp}$ is selfadjoint.
\end{theorem}
\begin{proof}
The fact that $A$ is a maximal symmetric relation, asserts by Proposition~\ref{caydac} that $\CZ i A$ is an isometry in $\cA B(\H)$. Then, by Theorem~\ref{Desconucnu},  there exists a unique reducing 
  subspace $\K$ for $\CZ i A$, such that $\CZ i A_{\K^\perp}$ is unitary
  and $\CZ i A_{\K}$ is a unilateral shift. Furthermore, Theorem~\ref{redsubincal} shows that $\K$
  reduces $A$. We follow the same reasoning of \eqref{eq:properties-Z-Nagy} to obtain
\begin{align*}
A=\oP{\CZ i {\CZ i A_{\K}}}{\CZ i {\CZ i A_{\K^{\perp}}}}\,.
\end{align*}
Therefore, Remark~\ref{re:Decom-into-relations} and Proposition~\ref{caydac} imply that  $A_{\K^{\perp}}=\CZ i {\CZ i A_{\K^{\perp}}}$
is selfadjoint and $A_{\K}=\CZ i {\CZ i A_{\K}}$ is symmetric, which certainly is elementary maximal.  To finish, we apply the properties of $Z$ transform and  \eqref{CtoT} in \eqref{carHhanm} to  obtain  \eqref{eq:reducing-maximal-symm}. Uniqueness is proven following  the same
lines of the proof of Theorem~\ref{Desconucnu01}. 
\end{proof}

\section{Example} \label{sec:Example}

Let us consider the Hilbert space of square-summable sequences $l_2(\N)$, with canonical basis $\{\delta_k\}_{k\in\N}$. We establish by $l_2(\fin)\subset l_2(\N)$, as the set of all sequences with only a finite number of nonzero entries. We define the linear operator  $\tilde A$, whose domain is given by
\begin{gather*}
\dom \tilde A:=\llb\sum_{k\in\N}(f_k-if_{k-1})\delta_k\,:\, \sum_{k\in\N}f_k\delta_k\in l_2(\fin)\rrb\,,
\end{gather*}
such that
\begin{gather*}
\tilde A\left(\sum_{k\in\N}(f_k-if_{k-1})\delta_k\right)=\sum_{k\in\N}(if_k-f_{k-1})\delta_k\,.
\end{gather*}

Let $A$ denote the closure of $\tilde A$. By a simple computation, one verifies that $S:=\CZ i A$ is the shift operator $S\delta_k=\delta_{k+1}$. Then, $A$ is elementary maximal and in according to Theorem~\ref{teo:elem-maximal}, it is unbounded and densely defined in $l_2(\N)$, with spectral properties \eqref{eq:spectral-properties-ME}. Inasmuch as  $\ran S=\oM{l_2(\N)}{\Span \{\delta_1\}}$, one has by  \eqref{CtoT} and  \eqref{poHs} that    
\begin{align}\label{eq:indices-elementary-A}
\begin{split}
\Span\{\delta_1\}&=\oM{l_2(\N)}{\ran \CZ i A}\\
&=\oM{l_2(\N)}{\ran(A-iI)}=\ker(A^*+iI)\,.
\end{split}
\end{align}
Subsequently, \eqref{eq:adjoint-elementary} yields 
\bea\label{eq:adj-elem-max-Exam}
A^*=\oP{A}{\Span\llb\vE{\delta_1}{-i\delta_1}\rrb}\,.
\eea
For simplicity of notation, we set $\K=\oM{l_2(\N)}{\Span \{\delta_1\}}$ and $
Y=\Span\llb\vE 0{\delta_1}\rrb$.

We now consider \beao B:=A\rE{\K}\,,\eeao which is a closed symmetric operator with $\CZ i B=S\rE{\K}$. Besides, it is straightforward to compute that $B=A\cap Y^*$ and since $Y$ is unidimensional,\begin{align}\label{eq:adjoint-auxB}
\begin{split}
B^*&=\left(A\cap Y^*\right)^*\\
&=-\left(\left(A\cap Y^*\right)^\perp\right)^{-1}\\
&=-\left(\cc{A^\perp\dotplus (Y^*)^\perp}\right)^{-1}\\
&=\cc{(-A^{-1})^\perp\dotplus(- (Y^*)^{-1})^\perp}\\
&=\cc{A^*\dotplus Y}=A^*\dotplus Y\,.
\end{split}
\end{align}

We observe that $\ran B\subset\K$. In this fashion, we may define the linear relation 
\beao
A_\infty:=\oP{B}{Y}\,,
\eeao
which is a closed symmetric extension of $B$, whit multivalued part $Y$. In view that the maximality of $A$ is equivalent to say $\dim \Nk {\cc\zeta} {A^*}=0$, $\zeta\in\C_-$ (\cf \cite{riossilva-expos}). Then, in view of   \eqref{eq:adjoint-auxB},
\begin{gather*}
\Nk {\cc\zeta}{A_{\infty}^*}\subset\Nk {\cc\zeta}{B^*}=\Nk {\cc\zeta}{A^*}\,,
\end{gather*}
wherefrom  $A_\infty$ is maximal. Following the same reasoning of \eqref{eq:indices-elementary-A} and \eqref{eq:adj-elem-max-Exam}, for $B$ in $\K$, one produces
\bea\label{eq:adj-B-decomposed}
B^*=\oP{B}{\Span\llb\vE{\delta_2}{-i\delta_2}\rrb}\,.
\eea 
One notes at once that $\K$ reduces $A_\infty$. Consequently,  Theorem~\ref{teo:invariante-adjunto}, \eqref{eq:reducing-distribution} and \eqref{eq:adj-B-decomposed} yield 
\begin{align}\label{eq:adjoint-A-infty}
(A_\infty)^*=\oP{\oP{B}{\Span\llb\vE{\delta_2}{-i\delta_2}\rrb}}{Y}\,.
\end{align}
Thus, $\dom\Nk{-i}{(A_\infty)^*}=\Span\{\delta_2\}$ and 
\begin{align*}
\bigoplus_{n=0}^\infty\CZ i{A_\infty}^n (\Span\{\delta_2\})&=
\bigoplus_{n=0}^\infty\CZ i{B}^n (\Span\{\delta_2\})\\&=
\bigoplus_{n=0}^\infty S^n(\Span\{\delta_2\})=\K\,.
\end{align*}
Therefore, by virtue of Theorem~\ref{descsisaem}, one concludes that $\K$ is the unique reducing subspace for $A_\infty$ in which $B$ is elementary maximal and $Y$ is selfadjoint in  $\Span\{\delta_1\}$. 


\begin{thebibliography}{10}

\bibitem{MR3178945}
Keshav~Raj Acharya, \emph{Self-adjoint extension and spectral theory of a
  linear relation in a {H}ilbert space}, ISRN Math. Anal. (2014), Art. ID
  471640, 5. \MR{3178945}

\bibitem{MR1255973}
N.~I. Akhiezer and I.~M. Glazman, \emph{Theory of linear operators in {H}ilbert
  space}, Dover Publications Inc., New York, 1993, Translated from the Russian
  and with a preface by Merlynd Nestell, Reprint of the 1961 and 1963
  translations, Two volumes bound as one. \MR{1255973 (94i:47001)}

\bibitem{MR0123188}
Richard Arens, \emph{Operational calculus of linear relations}, Pacific J.
  Math. \textbf{11} (1961), 9--23. \MR{0123188 (23 \#A517)}

\bibitem{acta15367}
Aharon Atzmon, \emph{Unicellular and non-unicellular dissipative operators},
  Acta scientiarum mathematicarum \textbf{57} (1993), no.~1-4, 45--54,
  Bibliogr.: p. 53-54.

\bibitem{MR3057107}
T.~Ya. Azizov, A.~Dijksma, and G.~Wanjala, \emph{Compressions of maximal
  dissipative and self-adjoint linear relations and of dilations}, Linear
  Algebra Appl. \textbf{439} (2013), no.~3, 771--792. \MR{3057107}

\bibitem{MR3050306}
Jussi Behrndt and Matthias Langer, \emph{Elliptic operators,
  {D}irichlet-to-{N}eumann maps and quasi boundary triples}, Operator methods
  for boundary value problems, London Math. Soc. Lecture Note Ser., vol. 404,
  Cambridge Univ. Press, Cambridge, 2012, pp.~121--160. \MR{3050306}

\bibitem{MR3705312}
Jussi Behrndt, Matthias Langer, Vladimir Lotoreichik, and Jonathan Rohleder,
  \emph{Quasi boundary triples and semi-bounded self-adjoint extensions}, Proc.
  Roy. Soc. Edinburgh Sect. A \textbf{147} (2017), no.~5, 895--916.
  \MR{3705312}

\bibitem{MR1192782}
M.~Sh. Birman and M.~Z. Solomjak, \emph{Spectral theory of selfadjoint
  operators in {H}ilbert space}, Mathematics and its Applications (Soviet
  Series), D. Reidel Publishing Co., Dordrecht, 1987, Translated from the 1980
  Russian original by S. Khrushch{\"e}v and V. Peller. \MR{1192782 (93g:47001)}

\bibitem{a2018BBMMSN}
B~Brown, M~Marletta, S~Naboko, and Ian Wood, \emph{The functional model for
  maximal dissipative operators: An approach in the spirit of operator knots},
  arXiv:1804.08963v1 (2018).

\bibitem{MR0477855}
Earl~A. Coddington, \emph{Extension theory of formally normal and symmetric
  subspaces}, American Mathematical Society, Providence, R.I., 1973, Memoirs of
  the American Mathematical Society, No. 134. \MR{0477855}

\bibitem{MR1873624}
G.~Corach, A.~Maestripieri, and D.~Stojanoff, \emph{Generalized {S}chur
  complements and oblique projections}, Linear Algebra Appl. \textbf{341}
  (2002), 259--272, Special issue dedicated to Professor T. Ando. \MR{1873624}

\bibitem{MR1631548}
Ronald Cross, \emph{Multivalued linear operators}, Monographs and Textbooks in
  Pure and Applied Mathematics, vol. 213, Marcel Dekker, Inc., New York, 1998.
  \MR{1631548}

\bibitem{MR3192426}
Bruno Despr\'{e}s, Lise-Marie Imbert-G\'{e}rard, and Ricardo Weder,
  \emph{Hybrid resonance of {M}axwell's equations in slab geometry}, J. Math.
  Pures Appl. (9) \textbf{101} (2014), no.~5, 623--659. \MR{3192426}

\bibitem{MR0361889}
A.~Dijksma and H.~S.~V. de~Snoo, \emph{Self-adjoint extensions of symmetric
  subspaces}, Pacific J. Math. \textbf{54} (1974), 71--100. \MR{0361889 (50
  \#14331)}

\bibitem{MR0250106}
R.~G. Douglas, \emph{On the operator equation {$S^{\ast} XT=X$} and related
  topics}, Acta Sci. Math. (Szeged) \textbf{30} (1969), 19--32. \MR{0250106}

\bibitem{MR3607473}
Marco Falconi, J\'{e}r\'{e}my Faupin, J\"{u}rg Fr\"{o}hlich, and Baptiste
  Schubnel, \emph{Scattering theory for {L}indblad master equations}, Comm.
  Math. Phys. \textbf{350} (2017), no.~3, 1185--1218. \MR{3607473}

\bibitem{MR2093073}
Mercedes Fernandez~Miranda and Jean-Philippe Labrousse, \emph{The {C}ayley
  transform of linear relations}, Proc. Amer. Math. Soc. \textbf{133} (2005),
  no.~2, 493--499. \MR{2093073}

\bibitem{MR3405898}
Alexander Figotin and Aaron Welters, \emph{Dissipative properties of systems
  composed of high-loss and lossless components}, J. Math. Phys. \textbf{53}
  (2012), no.~12, 123508, 40. \MR{3405898}

\bibitem{MR3864399}
Christoph Fischbacher, \emph{The nonproper dissipative extensions of a dual
  pair}, Trans. Amer. Math. Soc. \textbf{370} (2018), no.~12, 8895--8920.
  \MR{3864399}

\bibitem{MR3551233}
Christoph Fischbacher, Sergey Naboko, and Ian Wood, \emph{The proper
  dissipative extensions of a dual pair}, Integral Equations Operator Theory
  \textbf{85} (2016), no.~4, 573--599. \MR{3551233}

\bibitem{MR675952}
Paul~Richard Halmos, \emph{A {H}ilbert space problem book}, second ed.,
  Graduate Texts in Mathematics, vol.~19, Springer-Verlag, New York-Berlin,
  1982, Encyclopedia of Mathematics and its Applications, 17. \MR{675952}

\bibitem{MR2327982}
S.~Hassi, Z.~Sebesty\'en, H.~S.~V. de~Snoo, and F.~H. Szafraniec, \emph{A
  canonical decomposition for linear operators and linear relations}, Acta
  Math. Hungar. \textbf{115} (2007), no.~4, 281--307. \MR{2327982}

\bibitem{MR2150518}
Seppo Hassi, Mark Malamud, and Vadim Mogilevskii, \emph{Generalized resolvents
  and boundary triplets for dual pairs of linear relations}, Methods Funct.
  Anal. Topology \textbf{11} (2005), no.~2, 170--187. \MR{2150518}

\bibitem{MR725424}
I.~S. Kats, \emph{Linear relations generated by canonical differential
  equations}, Funktsional. Anal. i Prilozhen. \textbf{17} (1983), no.~4,
  86--87. \MR{725424}

\bibitem{MR0433235}
Thomas~L. Kriete, \emph{Canonical models and the self-adjoint parts of
  dissipative operators}, J. Functional Analysis \textbf{23} (1976), no.~1,
  39--94. \MR{0433235}

\bibitem{MR1464436}
Carlos~S. Kubrusly, \emph{An introduction to models and decompositions in
  operator theory}, Birkh\"auser Boston, Inc., Boston, MA, 1997. \MR{1464436}

\bibitem{MR2186014}
Graeme~W. Milton, Nicolae-Alexandru~P. Nicorovici, Ross~C. McPhedran, and
  Viktor~A. Podolskiy, \emph{A proof of superlensing in the quasistatic regime,
  and limitations of superlenses in this regime due to anomalous localized
  resonance}, Proc. R. Soc. Lond. Ser. A Math. Phys. Eng. Sci. \textbf{461}
  (2005), no.~2064, 3999--4034. \MR{2186014}

\bibitem{MR0104919}
R.~S. Phillips, \emph{Dissipative operators and hyperbolic systems of partial
  differential equations}, Trans. Amer. Math. Soc. \textbf{90} (1959),
  193--254. \MR{0104919}

\bibitem{MR2003221}
Heydar Radjavi and Peter Rosenthal, \emph{Invariant subspaces}, second ed.,
  Dover Publications, Inc., Mineola, NY, 2003. \MR{2003221}

\bibitem{riossilva-expos}
Josu\'e~I. Rios-Cangas and Luis~O. Silva, \emph{Dissipative extension theory
  for linear relations}, Expo. Math.
  \textbf{{DOI}:10.1016/j.exmath.2018.10.004} (2018).

\bibitem{arXiv:1806.07503}
\bysame, \emph{Perturbation theory for selfadjoint relations}, Annals of
  Functional Analysis \textbf{(forthcoming)} (2019), no.~arXiv:1806.07503.

\bibitem{MR2953553}
Konrad Schm{\"u}dgen, \emph{Unbounded self-adjoint operators on {H}ilbert
  space}, Graduate Texts in Mathematics, vol. 265, Springer, Dordrecht, 2012.
  \MR{2953553}

\bibitem{MR1037765}
B.~M. Solomyak, \emph{A functional model for dissipative operators. {A}
  coordinate-free approach}, Zap. Nauchn. Sem. Leningrad. Otdel. Mat. Inst.
  Steklov. (LOMI) \textbf{178} (1989), no.~Issled. Line\u{\i}n. Oper. Teorii
  Funktsi\u{\i}. 18, 57--91, 184--185. \MR{1037765}

\bibitem{MR2681000}
Laurian Suciu, \emph{Canonical decompositions induced by {$A$}-contractions},
  Note Mat. \textbf{28} (2008), no.~2, 187--202 (2010). \MR{2681000}

\bibitem{MR0058128}
B\'ela Sz.-Nagy, \emph{Sur les contractions de l'espace de {H}ilbert}, Acta
  Sci. Math. Szeged \textbf{15} (1953), 87--92. \MR{0058128}

\bibitem{MR2760647}
B{\'e}la Sz.-Nagy, Ciprian Foias, Hari Bercovici, and L{\'a}szl{\'o}
  K{\'e}rchy, \emph{Harmonic analysis of operators on {H}ilbert space}, second
  enlarged ed., Universitext, Springer, New York, 2010. \MR{2760647}

\bibitem{MR3350727}
A.~F.~M. ter Elst, Manfred Sauter, and Hendrik Vogt, \emph{A generalisation of
  the form method for accretive forms and operators}, J. Funct. Anal.
  \textbf{269} (2015), no.~3, 705--744. \MR{3350727}

\bibitem{MR3738818}
Ekin U\u{g}urlu and Dumitru Baleanu, \emph{On a completely non-unitary
  contraction and associated dissipative difference operator}, J. Nonlinear
  Sci. Appl. \textbf{10} (2017), no.~11, 5999--6019. \MR{3738818}

\bibitem{MR2926200}
Zhong Wang and Hongyou Wu, \emph{Dissipative non-self-adjoint
  {S}turm-{L}iouville operators and completeness of their eigenfunctions}, J.
  Math. Anal. Appl. \textbf{394} (2012), no.~1, 1--12. \MR{2926200}

\end{thebibliography}

\def\cprime{$'$} \def\lfhook#1{\setbox0=\hbox{#1}{\ooalign{\hidewidth
  \lower1.5ex\hbox{'}\hidewidth\crcr\unhbox0}}} \def\cprime{$'$}
  \def\cprime{$'$} \def\cprime{$'$} \def\cprime{$'$} \def\cprime{$'$}
  \def\cprime{$'$} \def\cprime{$'$}
\providecommand{\bysame}{\leavevmode\hbox to3em{\hrulefill}\thinspace}
\providecommand{\MR}{\relax\ifhmode\unskip\space\fi MR }
\providecommand{\MRhref}[2]{%
  \href{http://www.ams.org/mathscinet-getitem?mr=#1}{#2}
}
\providecommand{\href}[2]{#2}

\end{document}